\documentclass[12pt]{amsart}

\usepackage{amsfonts,amsthm,amsmath,amssymb,amscd}
\usepackage{graphics}
\usepackage{indentfirst}
\usepackage{cite}
\usepackage{latexsym}
\usepackage[dvips]{epsfig}

\newtheorem{theorem}{Theorem}[section]
\newtheorem{remark}{Remark}[section]

\newtheorem{lemma}[theorem]{Lemma}

\newcommand{\bt}{\begin{theorem}}
\newcommand{\bl}{\begin{lemma}}
\newcommand{\el}{\end{lemma}}
\newcommand{\et}{\end{theorem}}

\newcommand{\la}{\label}

\newcommand{\bn}{\begin{eqnarray}}
\newcommand{\en}{\end{eqnarray}}
\newcommand{\bnn}{\begin{eqnarray*}}
\newcommand{\enn}{\end{eqnarray*}}

\newcommand{\ba}{\begin{aligned}}
\newcommand{\ea}{\end{aligned}}
\newcommand{\be}{\begin{equation}}
\newcommand{\ee}{\end{equation}}

\newcommand{\no}{\nonumber \\}

\renewcommand{\la}{\label}

\begin{document}

\title[Inhomogeneous incompressible Navier-Stokes equations]
   {Global wellposedness for the 3D inhomogeneous incompressible Navier-Stokes equations}
\author{Walter Craig}
\address{Department of Mathematics and Statistics, McMaster
  University, 1280 Main Street West, Hamilton, Ontario L8S 4K1,
  Canada} 
\email{craig@math.mcmaster.ca}
\author{Xiangdi Huang}
\address{Academy of Mathematics and Systems Science, CAS, Beijing
  100190, P. R. China \ \ $\&$\ \ Department of Pure and Applied
  Mathematics, Graduate School of Information Sciences and Technology,
  Osaka University, Osaka, Japan} 
\email{xdhuang@amss.ac.cn}
\author{Yun Wang}
\address{Department of Mathematics and Statistics, McMaster
  University, 1280 Main Street West, Hamilton, Ontario L8S 4K1,
  Canada} 
\email{yunwang@math.mcmaster.ca}
\date{}

\thanks{The research of WC is supported in part by a Killam Research Fellowship, the
Canada Research Chairs Program and NSERC through grant number
238452--11. The research of YW is supported in part by a Canada
Research Chairs Postdoctoral Fellowship at McMaster University.}

\maketitle

\begin{abstract}This paper addresses the three-dimensional
Navier-Stokes equations for an incompressible fluid whose density is
permitted to be inhomogeneous. We establish a theorem of global
existence and uniqueness of strong solutions for initial data with
small $\dot{H}^{\frac12}$-norm, which also satisfies a
natural compatibility condition.  A key point of the theorem is that
the initial density need not be strictly positive.  


AMS: 35Q35, 35B65, 76N10
\end{abstract}


\section{Introduction}


In a number of important applications to fluid mechanics, such as in
geophysical fluid dynamics, the Navier-Stokes equations are called
upon to describe situations in which a fluid is inhomogeneous with
respect to density, however it is essentially incompressible to a
great extent. Such cases occur within the outer core and mantle of the
earth, for example, or else within the atmosphere in a regime of flows
for which dynamic effects of compressibility do not play a principle
r\^ole. Ignoring any possible buoyancy or other body forces, flows of
this character satisfy the following system of inhomogeneous
incompressible Navier-Stokes equations;  
\begin{equation}\label{ns}
  \left\{ \begin{array}{l}
     \partial_t\rho + {\rm div} (\rho u) = 0 ~, \\
     \partial_t(\rho u) + {\rm div}(\rho u\otimes u) - \mu \Delta u +
     \nabla P = 0 ~, \\
     {\rm div}~u = 0 ~.
  \end{array} \right.
\end{equation}
It is a standard mathematical model to pose the initial and boundary
value problem, giving the following data;   
\begin{equation}\label{initial}
   \begin{array}{c}
     (\rho, u)|_{t=0} = (\rho_0, u_0) \quad \mbox{in}\ \Omega~, \\
     u = 0 \quad \mbox{on}\ \partial \Omega~, \qquad u(x,
     \cdot)\rightarrow 0\ \quad \mbox{as}\ |x|\rightarrow
     \infty ~.
   \end{array} 
\end{equation}
Here $\rho$, $u$ and $P$ denote the density, velocity and pressure of
the fluid respectively, and $\mu$ is the positive viscosity, which is assumed to be
a constant. In this paper, we consider  $\Omega$ is either a smooth
bounded domain in $\mathbb{R}^3$ or the whole space $\mathbb{R}^3$.  
 
For the initial density with positive lower bound, the inhomogeneous
equations \eqref{ns} have been studied in the sixties and seventies by
the Russian school, see \cite{AKM, Kazhikov, LS} and the many references
that they contain. It was proved that a unique strong solution exists 
locally for arbitrary initial data. Moreover, these papers also establish
global wellposedness results for small solutions in dimension $N \ge
3$, while for the two dimensional case they establish the existence of 
large strong solutions. More recently, there have been many subsequent 
contributions to this theory, obtaining global wellposedness for 
initial data belonging to certain scale invariant spaces, see for example
\cite{AP,Danchin,GHZ}.   

On the other hand, for initial data which permits regions of vacuum,
{\it i.e.} regions where the density $\rho$ vanishes on some 
set, the problem becomes much more involved. Authors including 
DiPerna and Lions\cite{DL, Lions} prove the global existence of weak
solutions to \eqref{ns} in any space dimension, see also
\cite{Simon,Desjardins} and the references in their work. As pointed
in \cite{Choe-Kim}, the major difficulty for existence of strong
solutions which admit regions of vacuum is the lack of an appropriate
estimate for $\partial_t u$, since $\partial_t u$ only appears in the
momentum equation with a possibly degenerate coefficient $\rho$. One
way to avoid this is to estimate $\nabla \partial_tu$ instead, for which
one pays the price of being required to impose a compatibility condition. 
In fact, Choe-Kim\cite{Choe-Kim} prove the following theorem,  
\bt\label{thm-Choe-Kim} 
Assume that the initial data $(\rho_0, u_0) $ satisfies 
\begin{equation}
   0 \leq \rho_0 \in L^{\frac{3}{2}} \cap H^2~, \qquad 
     u_0 \in D_0^{1,2} \cap D^{2, 2}~,
\end{equation}
and the compatibility condition
\begin{equation}
    \mu \Delta u_0 - \nabla P_0 = \sqrt{\rho_0} g_0 ~, \qquad {\rm div}~u_0 =
    0 \quad \mbox{in} \ \Omega~, 
\end{equation}
for some $(P_0, g_0) \in D^{1, 2} \times L^2$. Then there exists a
time $T > 0$ and a unique strong solution $(\rho, u, P)$ to the
initial boundary value problem \eqref{ns}-\eqref{initial} satisfying  
\begin{equation}\no
   \begin{array}{c} 
     \rho \in L^\infty(0, T; L^{\frac32} \cap H^2), \quad u \in C([0, T]~;
       D^{1,2}_{0, \sigma}\cap D^{2, 2} ) \cap L^2(0, T; D^{3,2})~, \\ 
     \partial_t u \in L^2(0, T; D^{1,2}_{0, \sigma})~, \quad \mbox{and} \quad
     \sqrt{\rho} \partial_t u \in L^\infty(0, T; L^2)~.  
\end{array}
\end{equation}
\et

It is also interesting to study the regularity criterion for such
strong solutions, see\cite{Cho-Kim, Kim} and its references. In
particular, Kim\cite{Kim} proved that if $T^*$ is the blowup time of
a local strong solution, then necessarily 
\begin{equation}\label{blowup-Kim}
   \int_0^{T^*} \|u(s)\|_{L^{q, \infty}}^p \, ds = \infty~, \quad 
   \mbox{for any} \ (p, q)~, \quad \frac{2}{p} + \frac{3}{q} = 1 ~, 
   \ 3<q \leq \infty ~, 
\end{equation}
where $L^{q, \infty}$ is the usual weak Lebesgue space. 
This is effectively a criterion of regularity, and is similar to the
Serrin criterion\cite{Serrin} for {\it a priori} regularity condition
for the homogeneous Navier-Stokes equations. 

However, like the usual homogeneous Navier-Stokes equations, the
question of global existence of strong solutions remains open for
large initial data. A partial result due to Kim\cite{Kim} asserts that
if $\|u_0\|_{D^{1,2}}$ is small enough, then a strong solution exists
globally in time and is unique. In this paper, we generalize this 
result to  the case of data within a
class of critical Sobolev spaces.   

There is a scaling invariance of the system \eqref{ns}, namely
if $(\rho, u,  P)$ is a solution associated with the initial data
$(\rho_0, u_0)$, then  
\[
   (\rho_\lambda, \ u_\lambda, P_\lambda) = (\rho(\lambda^2t, \lambda x)~, 
   \lambda u (\lambda^2 t, \lambda x), \lambda^2 P(\lambda^2 t, \lambda x))
\] 
is a solution to \eqref{ns}-\eqref{initial} associated with the
initial data $(\rho_0(\lambda x), \lambda u_0(\lambda x))$. A scale
invariant function space $H^*$ is one for which 
$\|(\rho_\lambda, u_\lambda)\|_{H^*} = \|(\rho_1, u_1)\|_{H^*}$. For example, 
scaling invariance implies that 
\[
    \|u\|_{\dot{H}^{\frac12}} = \|u_\lambda\|_{\dot{H}^{\frac12}} ~, 
\] 
which is our principle example. Our main result in this paper
is that if $\|u_0\|_{\dot{H}^{\frac12}}$ is small enough, then
a strong solution exists globally in time. More precisely,  

\bt\label{main-result} Assume the conditions in Theorem~\ref{thm-Choe-Kim}. 
Then there exists a small constant $\varepsilon$ depending on $\bar{\rho}
= \|\rho_0\|_{L^\infty}$, $\mu$ and the domain $\Omega$, such that if  
\begin{equation}\label{small-condition}
   \|u_0\|_{\dot{H}^{\frac12}} \leq \varepsilon ~, 
\end{equation}
then the unique local strong solution constructed in Theorem
\ref{thm-Choe-Kim} exists globally in time. Moreover, this global
strong solution satisfies the following decay properties: 
\be \label{decay}
   \|\nabla u\|_{L^2} \leq M t^{-\frac12} ~,
\ee
where $M$ depends on $\mu$, $\Omega$, $\bar{\rho}$, and $\int \rho_0
|u_0|^2 dx$.  
\et

The global result above relies on the smallness of the scale invariant
norm of the initial data. In terms of the well known result of
Fujita \& Kato\cite{Fujita-Kato} for the homogeneous Navier-Stokes 
equations, our result can be considered as its generalization to the
inhomogeneous case. The homogeneous Navier-Stokes equations have been
a much more active topic of inquiry in this direction, see 
\cite{Cannone,CheminGallagherPaicu,Koch-Tataru} and their references. 
The key ideas leading to these
results come from a variety of refined estimates of Stokes semigroup 
combined with the contraction principle for a Picard iteration scheme. 
However, in the present work one can not directly apply the same methods, 
due to the presence of the additional factor $\rho$ before the term 
$\partial_t u$. Instead we pursue a different strategy of proof, based 
on the method of energy estimates, and using the parabolic property of 
the system of equations. The main idea goes as follows: we first assume 
that the scale-invariant quantity $\|\nabla u\|_{L^4_t L^2_x}$ is less than an {\it a priori} 
constant bound $2$; then using this assumption we then prove, under the 
smallness assumption on the initial data, that in fact the quantity 
is less than $1$. Since $\|\nabla u\|_{L^4_t L^2_x}$ is initially $0$, which is less than $1$, 
then it remains less than $1$. On the other hand, the
boundedness of $\|\nabla u\|_{L^4_t L^2_x}$ implies the local solution can be extended, according to 
some a priori estimates. 

The proof in this paper is inspired by a similar result for 3D
compressible Navier-Stokes equations due to Huang, Li \& Xin\cite{HLX}, 
in which they proved the global wellposedness of classical solutions
with small initial $H^{\frac12+ \varepsilon}$-norm.    

\noindent
{\it Remarks:} Compared to the global existence results in 
\cite{AP,Danchin,GHZ}, for which the initial density is required to be
a small perturbation of a constant, our result does not require an 
assumption of smallness on the density variations, and in fact it 
allows for the presence of regions of vacuum.  

In the case of periodic boundary conditions on a torus ${\mathbb T}^d$, 
the same result can be proved if the intial data has zero momentum,
following the strategy as for the case of a bounded domain.  
Indeed, the same conclusion holds for classical solutions if the
initial data are regular enough and satisfy some higher-order
compatibility conditions. The proof of this statement is almost the
same except that some additional high-order estimates are
required. These estimates have been given in \cite{Kim}.   
In the case in which $\Omega$ is a smooth bounded domain in $\mathbb{R}^2$,
it is already known that global strong solutions exist without a smallness
assumption; this was proved first for the case without vacuum
\cite{AKM} and subsequently for the case with vacuum\cite{Huang-Wang}.  

Our result holds in particular in the case of a constant density $\rho = 1$. 
It therefore gives a new proof of the well known Fujita-Kato theorem \cite{Fujita-Kato} ,
which uses only energy estimate methods in the argument.
\begin{theorem}\label{F-K}
Assume $u_0\in H^1(\mathbb{R}^3)$ and ${\rm div}~u_0 = 0$. There
exists a positive constant $\varepsilon$, such that if
$\|u_0\|_{\dot{H}^{\frac12}} \leq \varepsilon$, then  
there is a global strong solution to the homogeneous Navier-Stokes equations with the property
that 
\[
   u\in C([0, \infty); H^1) \cap L^2_{loc}(0, \infty; H^2) ~.
\]
\end{theorem}

This article is organized as follows: section 2 sets the notation, and
contains several definitions and basic lemmas. In section 3 we give the
proof for the case $\Omega=\mathbb{R}^3$, while the proof for the case
of a bounded domain is presented in section 4. 

\section{Notation and interpolation lemmas}

The homogeneous and inhomogeneous Sobolev spaces are defined in the
standard way. For $1 \leq r \leq \infty$ and $k\in \mathbb{N}$, 
\[
  \begin{array}{l}
   L^r = L^r(\Omega), \\ 
   D^{k, r}= \{ u \in L_{loc}^1 : \ 
      \|\nabla^k u\|_{L^r} < \infty \}~, \quad \|u\|_{D^{k, r}} =
      \|\nabla^k u\|_{L^r} ~,  \\ 
   W^{k, r} = L^r\cap D^{k, r} ~, \quad H^k = W^{k, 2} ~, \quad 
   D_0^{k, r} = \overline{C_0^\infty} \ \mbox{closure in the norm of} \ D^{k, r}~.
   \end{array}
\]
In particular when $\Omega$ is a bounded domain then $D_0^{1, 2} = H_0^1$. Set 
\[
   C^\infty_{0, \sigma} = \{u \in C_0^\infty:\ {\rm div}~u =0\} ~,
\]
and set $D_\sigma^{1,2}= \overline{C_{0, \sigma}^\infty} $, with the closure
taken in the norm of $D^{1,2}$.  The fractional-order homogeneous Sobolev
space $\dot{H}^s(\mathbb{R}^3)$ are defined as the space of tempered
distributions $u$ over $\mathbb{R}^3$  for which the Fourier transform
$\mathcal{F}u$ belongs to $L_{loc}^1(\mathbb{R}^3)$ and which satisfy
\be \nonumber 
    \|u\|_{\dot{H}^s(\mathbb{R}^3)}^2 := \int_{\mathbb{R}^3}
    |\xi|^{2s} |\mathcal{F} u(\xi)|^2 d\xi < \infty ~. 
\ee
And $\dot{H}^s(\Omega)$ is the restriction of $\dot{H}^s(\mathbb{R}^3)$ to the domain $\Omega$.
We refer to $\dot{H}^s(\Omega)$ by $\dot{H}^s$, independent of whether $\Omega$ is the whole space or not.

Some well known interpolation results which relate Lorentz spaces and 
the classical Sobolev spaces are presented, which can be found in \cite{BL, Tartar,
Tribel}. For every $0< \theta <1 $, $1 \leq q \leq \infty$, and
normed spaces $X_0, X_1$,  we use the notation $[X_0, X_1]_{\theta,q}$ 
to denote the real interpolation space between $X_0$ and $X_1$. 
\begin{lemma}\label{lemma-Besov-interpolation}Let $s, s_0, s_1 \in
    \mathbb{R}$.  If $s = (1-\theta) s_0 + \theta s_1$ with some $0<\theta<1$, then
  \be \nonumber
    [\dot{H}^{s_0}, \dot{H}^{s_1} ]_{\theta, 2} =\dot{H}^s.
  \ee
And the following inclusion relation holds, if $1\leq q_1 \leq q_2\leq  \infty$, then
\be  \nonumber
 [\dot{H}^{s_0}, \dot{H}^{s_1}]_{\theta, q_1} \subseteq  [\dot{H}^{s_0}, \dot{H}^{s_1}]_{\theta, q_2}. 
\ee
\end{lemma}

\begin{lemma}\label{lemma-Lorentz-interpolation}Let $0<\theta <1$,
  $1<p_0, p_1, p < \infty$, $1\leq q_0, q_1, q\leq \infty$, and
  $\frac{1}{p} = \frac{1-\theta}{p_0} + \frac{\theta}{p_1}$, then  
  \be \label{Lorentz-interpolation}
    [L^{p_0, q_0}, L^{p_1, q_1}]_{\theta, q} = L^{p, q} ~.
  \ee
  This identification of the interpolation spaces
  \eqref{Lorentz-interpolation} is also valid if  $L^{p_1, q_1}$ is
  replaced by by $L^\infty$ ($L^{p_0, q_0}$, $L^{p_1, q_1}$ and $L^{p,
  q}$ are Lorentz spaces, whose definition can be found in
  \cite{BL}). When $q = p$, $L^{p, p}$ is 
  the Sobolev $L^p$ space.  
\end{lemma}

The above interpolation lemmas clearly generalize to vector-valued
Lorentz spaces. For example,  
\begin{lemma}\label{lemma-vector-Lorentz-interpolation} 
  Let $0<\theta <1$, $1<p_0, p_1 < \infty$, $1\leq q_0, q_1, q\leq
  \infty$, and $\frac{1}{p} = \frac{1-\theta}{p_0} +
  \frac{\theta}{p_1}$. Then 
  \be \label{vector-Lorentz-interpolation}
     [L^{p_0, q_0}(0, T; L^2), L^{p_1, q_1}(0, T; L^2)]_{\theta, q} =
     L^{p, q}(0, T; L^2) ~. 
  \ee
  The statement \eqref{vector-Lorentz-interpolation} is also valid
  after replacing $L^{p_1, q_1}(0, T; L^2)$ by $L^\infty(0, T; L^2)$.  
\end{lemma}

Bounds for linear operators on interpolation spaces are the concern of
the next lemma. 
\begin{lemma}\label{lemma-operator-interpolation}
  Suppose that $E_0, E_1, F_0, F_1$ are Banach spaces and 
  $0< \theta <1$, $1\leq q \leq \infty$, and that $A$ is a linear
  operator from $E_0 +E_1$ into $F_0 + F_1$. If $A$ maps $E_0$ into
  $F_0$, with $\|A w \|_{F_0} \leq L_0 \|w\|_{E_0}$ for all $w\in
  E_0$, and maps $E_1$ into $F_1$, with $\|A w\|_{F_1} \leq L_1
  \|w\|_{E_1}$ for all $w \in E_1$. Then $A$ is a continous linear
  operator from $[E_0, E_1]_{\theta, q} $ into $[F_0, F_1]_{\theta,q}$ 
  and one has  
  \be\nonumber
    \|Aw\|_{[F_0, F_1]_{\theta, q}} \leq L_0^{1-\theta} L_1^{\theta}
    \|w\|_{[E_0, E_1]_{\theta, q}}~. 
  \ee
\end{lemma}

In this paper, the constants $C, C_1, C_2, C_3, C_4, C_5, \cdots$ may
depend upon $\mu, \bar{\rho}, \Omega$ and as conventional in analysis,
they may change from line to line. 


\section{Proof of Theorem \ref{main-result} for $\Omega= \mathbb{R}^3$}
In this section, we focus on the case $\Omega=\mathbb{R}^3$.

\subsection{A Priori Estimates}
This subsection establishes several {\it a priori} estimates for
strong solutions to the Cauchy problem, which will play a key role in
extending local strong solutions to global ones.   

 Given a strong solution $(\rho, u , P)$ on $\mathbb{R}^3 \times [0, T]$, define 
\be \la{A}
   A(t) = \|\nabla u\|_{L^{4}(0, t;\  L^2)} ~,\ \ \ \ \ 0\leq t \leq T~.
\ee

\begin{theorem}\label{keypro}
Under the assumptions of Theorem \ref{thm-Choe-Kim}, there exists a 
positive constant $\varepsilon$ depending only on $\mu$ and $\bar{\rho}$, 
such that if $\|u_0\|_{\dot{H}^{\frac12}}  \leq \varepsilon$, 
and $(\rho, u, P)$ is a strong solution to \eqref{ns}-\eqref{initial}, 
satisfying  
\be\label{assumption} 
   A(T)  \leq 2,
\ee
then it in fact holds that
\[
  A(T)  \leq 1.
\]
\end{theorem}
The remainder of this subsection consists in proving this key result.

\begin{lemma}\label{lemma-3-1}
Let $(\rho, u, P)$ be a strong solution of \eqref{ns}-\eqref{initial}. 
Then for all $0\leq t \leq T$, 
\be \label{Eqn:TransportProperty}
   \|\rho(t)\|_{L^\infty} = \|\rho_0\|_{L^\infty} = \bar{\rho} ~.
\ee
\end{lemma}

\begin{proof}
The mass equation is in fact a transport equation, owing to the fact that
${\rm div} (u) =0$, from which \eqref{Eqn:TransportProperty} follows. 
\end{proof}


\begin{lemma}\label{lemma-3-3}Let $(\rho, u, P)$ be a strong solution
  of \eqref{ns}-\eqref{initial}. Suppose that
  \be \nonumber
A(T) \leq 2,
\ee
 then there exists some constant $C_1$ depending on $\mu, \bar{\rho}$, such that
  \be \label{3-3}
    \|\nabla u\|_{L^4(0, T; L^2 )} \leq C_1 \|u_0\|_{\dot{H}^{\frac12}} ~,
  \ee
and it holds that 
\be \la{3-3-add}
\sup_{t\in [0, T]}  t \|\nabla u (t)\|_{L^2}^2 \leq C \int \rho_0 |u_0|^2 \, dx.
\ee
Note that $C_1$ does not depend on $T$.
\end{lemma}

\begin{proof}First, let's consider the following linear Cauchy problem
for $(w, \tilde{P})$,  
\be \la{Cauchy-system}
    \left\{ \ba
  & \rho \partial_t w - \mu \Delta w + (\rho u\cdot \nabla )w + \nabla
  \tilde{P} = 0 ~,  \\
  & {\rm div}~w=0 ~,  \\
  & w(x, 0) = w_0(x) ~.
\ea \right.
\ee
Suppose that $w_0$ satisfies the conditions assumed for $u_0$ as
prescribed in Theorem \ref{thm-Choe-Kim}, then the existence and
uniqueness of strong solution to \eqref{Cauchy-system} has been proved
in \cite{Choe-Kim}. Straightforward energy estimates tell that
\be\la{energy-3}
  \frac12 \int  \rho |w(T)|^2 \, dx +  \mu \int_0^T  \|\nabla w\|_{L^2}^2 \,
  dt  \leq  \frac12 \int \rho_0 |w_0|^2 \, dx \leq C(\bar{\rho})
  \|w_0\|_{L^2}^2 ~. 
\ee

Multiplying $\eqref{Cauchy-system}_1$ by $\partial_t w$ and integrating over
$\mathbb{R}^3$, one gets by Sobolev embedding that 
\be  \la{3-3-1} \ba
  \int \rho |\partial_t w|^2 \, dx & 
    + \frac{\mu}{2} \frac{d}{dt} \int |\nabla w|^2 \, dx 
    = - \int (\rho u\cdot \nabla ) w \cdot w_t \, dx \\
    \leq & C(\bar{\rho}) \|u\|_{L^6}  \|\nabla w\|_{L^3} \|\sqrt{\rho} \partial_t w\|_{L^2} \\ 
    \leq &  \frac14 \int \rho |\partial_t w|^2 \,  dx + C \|\nabla u\|_{L^2}^2 \|\nabla w\|_{L^2}  \|\nabla w\|_{L^6} ~. 
\ea \ee

Notice that the momentum equation can be written as
\be \nonumber
-\mu \Delta w + \nabla \tilde{P}= -\rho \partial_t w - (\rho u\cdot \nabla) w,
\ee
where the left handside is viewed as the Helmholtz-Weyl decomposition of the right one. 
From this equation,
\be \nonumber
\mu \| \Delta w \|_{L^2} \leq \|\rho \partial_t w\|_{L^2} + \|(\rho u\cdot \nabla ) w \|_{L^2}.
\ee
It follows from Calder\'on-Zygmund inequality and the Sobolev embedding theorem that
\be \la{3-3-2} \ba
& \|\nabla w\|_{L^6} \leq C \|\nabla^2 w\|_{L^2}    \\
\leq & C \|\rho \partial_t w\|_{L^2} + C \|(\rho u\cdot \nabla ) w\|_{L^2} \\
\leq & C \|\rho \partial_t w\|_{L^2} + C \|\nabla u\|_{L^2} \|\nabla w\|_{L^2}^{\frac12} \|\nabla w\|_{L^6}^{\frac12},
\ea \ee
which, together with Young's inequality, implies that
\be \la{3-3-3} 
\ba
\|\nabla w\|_{L^6} \leq C \|\nabla^2 w\|_{L^2}  \leq C \|\rho \partial_t w\|_{L^2} + C \|\nabla u\|_{L^2}^2 \|\nabla w\|_{L^2}.
\ea
\ee

Inserting \eqref{3-3-3} into \eqref{3-3-1}, we get 
\be \la{3-3-4}  
\frac12 \int \rho |\partial_t w|^2 \, dx 
    + \frac{\mu}{2} \frac{d}{dt} \int |\nabla w|^2 \, dx 
\leq C \|\nabla u\|_{L^2}^4 \|\nabla w\|_{L^2}^2.
\ee
By Gronwall's inequality and the assumption $A(T)\leq 2$,
\be \la{3-3-5}
\int_0^t \int \rho |\partial_t w(s)|^2 \, dx ds + \|\nabla w(t)\|_{L^2}^2 \leq C e^C \|\nabla w_0\|_{L^2}^2.
\ee

For fixed $(\rho, u)$, the map from $w_0$ to $\nabla w(t)$ is linear. Furthermore, by Lemmas \ref{lemma-Besov-interpolation} and \ref{lemma-vector-Lorentz-interpolation}, 
$L^4(0, T; L^2) =
\left[ L^{2}(0, T; L^2 ), \ L^\infty(0, T;  L^2)\right]_{\frac12, 4} $ and
$\dot{H}^{\frac12} \subseteq  \left[L^2, \dot{H}^1\right]_{\frac12, 4}  $, then one gets,
upon combining the estimate \eqref{energy-3} and \eqref{3-3-5}, 
and using Lemma \ref{lemma-operator-interpolation}, that 
\be \la{3-3-6}
   \|\nabla w\|_{L^4(0, T; L^2)} \leq C_1  \|w_0\|_{\dot{H}^{\frac12}} ~.
\ee
Notice that $C_1$ does not depend on $T$, since one can scale $[0, T]$
to $[0, 1]$.  Consequently, 
\be \nonumber
   \|\nabla u\|_{L^4(0, T; L^2)} \leq C_1 \|u_0\|_{\dot{H}^{\frac12}}.
\ee

Regarding the decay of $\|\nabla u\|_{L^2}$, upon multiplying \eqref{3-3-4} by $t$, one gets that
\be \nonumber
t\int \rho |\partial_t u|^2 \, dx +  \frac{d}{dt} \int t |\nabla u|^2 \, dx \leq \int |\nabla u|^2 dx 
+ Ct  \|\nabla u\|_{L^2}^4 \|\nabla u\|_{L^2}^2.
\ee
If $A(T)\leq 2$, then by Gronwall's inequality,
\be \la{3-3-7} \ba
 &\int_0^T t \|\sqrt{\rho}   \partial_t u \|_{L^2}^2\, dt  + \sup_{t\in [0, T]} t\|\nabla u(t)\|_{L^2}^2 \\
&\leq C \int_0^T \|\nabla u\|_{L^2}^2 \, dt \leq  C \int \rho_0 |u_0|^2 \, dx,
\ea
\ee
where the basic energy inequality for $u$ is utilized.

\end{proof}


\begin{proof}[Proof of Theorem~\ref{keypro}] 
The conclusion of Theorem \ref{keypro} will follow if we let 
$\varepsilon \leq 1/C_1$, which is given in 
Lemma~\ref{lemma-3-3}. 
\end{proof}

\begin{remark}According to the proof of Lemma \ref{lemma-3-3}, the $\dot{H}^{\frac12}$-norm of $u_0$ can 
be replaced by $\dot{B}^{\frac12}_{2, 4}$-norm, which will make the main result more refined.
\end{remark}

\vspace{4mm} From this point on, we will consider only the small data problem,
assuming that the initial data satisfies the condition 
$\|u_0\|_{\dot{H}^{\frac12} } \leq \varepsilon$, as in 
Theorem~\ref{keypro}.  
The notation $\bar{C}$ is used to denote a positive constant, which
may depend on $T$ and the initial data, and it may change from line to
line. The computation is standard, with respect to that in \cite{Kim},
but we sketch it here for completeness. From the proof of
Lemma \ref{lemma-3-3}, we know that   
\begin{equation}\la{3-3-8}
    \int_0^T \|\sqrt{\rho} \partial_t u\|_{L^2}^2\,  dt  
    + \mu \|\nabla u\|_{L^\infty(0, T; L^2)}^2 \leq C \|\nabla
    u_0\|_{L^2}^2 \leq \bar{C} ~.
\end{equation}


\begin{lemma}[{Estimates for $\|\sqrt{\rho} \partial_t u\|_{L^2}$ and
    $\|\nabla u\|_{H^1}$}] \label{lemma-3-6} 
  Under the assumption of Theorem~\ref{keypro}, we have 
\begin{equation}\label{3-6}
   \sup_{t\in [0, T]} \left[ \|\sqrt{\rho} \partial_t u\|_{L^2}^2 
   + \|\nabla u\|_{H^1}^2\right] 
   + \int_0^T \|\nabla \partial_t u\|_{L^2}^2 \, dt \leq \bar{C} ~.
\end{equation}
\end{lemma}

\begin{proof} Differentiating the momentum equation with respect to
  $t$, multiplying by $\partial_t u$, and then integrating over
  $\mathbb{R}^3$, one can obtain that  
\begin{equation}\label{3-6-1}
   \ba  & \frac12 \frac{d}{dt} \int \rho |\partial_t u|^2 \, dx 
     + \mu \int |\nabla \partial_t u|^2\,  dx  \\[3mm]
     = & -2\int \rho u\cdot \nabla \partial_t u \cdot \partial_t u \, dx 
     - \int \rho u \cdot \nabla (u \cdot \nabla u\cdot \partial_t u) \, dx 
     - \int (\rho \partial_t u \cdot \nabla ) u \cdot \partial_t u \, dx~.
\ea
\end{equation}

It follows from Sobolev embedding theorem and Gagliardo-Nirenberg inequality that
\be\la{3-6-2} \ba
&  -2 \int \rho u \cdot \nabla \partial_t u  \cdot \partial_t u \, dx \\
&\  \leq   C \|\nabla u\|_{L^2} \|\nabla \partial_t u\|_{L^2}^{\frac32} \|\sqrt{\rho} \partial_t u\|_{L^2}^{\frac12} \\
 &\  \leq  \frac{\mu}{8}\|\nabla \partial_t u\|_{L^2}^2 + C \|\nabla u\|_{L^2}^4 \|\sqrt{\rho} \partial_t u\|_{L^2}^2 ~.
\ea \ee
By Sobolev embedding and the estimate \eqref{3-3-3} , the second
term can be estimated,
\be\la{3-6-3} \ba
  -  \int & \rho u\cdot \nabla(u\cdot \nabla u \cdot \partial_t u) \, dx \\
   \leq &  \int \rho |u| |\nabla u|^2  |\partial_t u| \, dx 
  +  \int \rho |u|^2 |\nabla^2 u| |\partial_t u | \, dx 
  +  \int \rho |u|^2 |\nabla u| |\nabla \partial_t u| \, dx \\ 
  \leq & C   \|u\|_{L^6} \|\nabla u\|_{L^2} 
    \|\nabla u \|_{L^6} \|\partial_t u\|_{L^6}  
  + C \|u\|_{L^6}^2 \|\nabla^2  u\|_{L^2} 
    \|\partial_t u\|_{L^6} \\ 
  & \quad + C \| u\|_{L^6}^2 \|\nabla u\|_{L^6}
  \|\nabla \partial_t u\|_{L^2}  \\ 
  \leq &  C  \|\nabla u\|_{L^2}^4 \|\nabla^2 u \|_{L^2}^2 
  + \frac{\mu}{8}  \|\nabla \partial_t u \|_{L^2}^2 \\
\leq & C  \|\nabla u\|_{L^2}^4 \|\rho \partial_t u \|_{L^2}^2 + C  \|\nabla u \|_{L^2}^{10}
  + \frac{\mu}{8}  \|\nabla \partial_t u \|_{L^2}^2~.
\ea\ee
For the third term on the right handside of \eqref{3-6-1}, utilizing
Gagliardo-Nirenberg inequality and \eqref{3-3-3},  
\be\la{3-6-4}
\ba
  -  \int & (\rho \partial_t u \cdot \nabla ) u \cdot \partial_t u \, dx 
   \leq C   \|\sqrt{\rho} \partial_t u\|_{L^2} \|\nabla u\|_{L^6}
     \|\sqrt{\rho} \partial_t u\|_{L^3}  \\ 
   \leq & C \|\sqrt{\rho
   } \partial_t u \|_{L^2}^{\frac32} 
     \|\nabla \partial_t u\|_{L^2}^{\frac12}  \left(  \|\rho \partial_t u\|_{L^2}  + \|\nabla u\|_{L^2}^3 \right) \\
   \leq & C  \|\sqrt{\rho} \partial_t u\|_{L^2}^{\frac{10}{3}} + C \|\sqrt{\rho}  \partial_t u\|_{L^2}^2 \|\nabla u\|_{L^2}^4  + \frac{\mu}{8} 
      \|\nabla \partial_t u\|_{L^2}^2 ~.
\ea
\ee

Collecting all the estimates \eqref{3-6-2}-\eqref{3-6-4}, one gets that
\begin{equation}\la{3-6-5}
\ba
   &  \frac12 \frac{d}{dt} \int \rho |\partial_t u|^2 \, dx + \frac{\mu}{2} \int |\nabla \partial_t u|^2 \, dx  \\ 
      &\ \ \ \leq  C \|\nabla u\|_{L^2}^4 \|\sqrt{\rho} \partial_t u\|_{L^2}^2 
   + C \|\sqrt{\rho} \partial_t u\|_{L^2}^{\frac43} \|\sqrt{\rho} \partial_t u \|_{L^2}^2+ C\|\nabla u\|_{L^2}^{10}  ~.
\ea
\end{equation}

Utilizing Gronwall's inequality and the estimate \eqref{3-3-8}, we obtain that
\begin{equation}\label{3-6-6}
   \sup_{t\in [0, T]} \|\sqrt{\rho} \partial_t u\|_{L^2}^2 
   + \int_0^T \|\nabla \partial_t u\|_{L^2}^2\,  dt \leq \bar{C}.
\end{equation}

According to the estimate \eqref{3-3-3},  \eqref{3-6-6} implies that
\begin{equation}\label{3-6-7}
     \sup_{t\in [0, T]}\|\nabla^2 u \|_{L^2} 
     \leq C \sup_{t\in [0, T]} \|\sqrt{\rho} \partial_t u\|_{L^2}  + C\|\nabla u\|_{L^2}^3 \leq \bar{C} ~.
\end{equation}
It completes the proof.
\end{proof}


\begin{lemma}[{Estimate for $\|\nabla \rho\|_{H^1}$}]
\label{lemma-3-7}
 Under the assumptions of Theorem~\ref{keypro}, we know that
\begin{equation}\label{3-7}
   \sup_{t\in [0, T]} \|\nabla \rho\|_{H^1} + \int_0^T \|\nabla^3 u\|_{L^2}^2\, dt 
   \leq \bar{C} ~.
\end{equation}
\end{lemma}

\begin{proof}
Differentiating the mass equation with respect to $x_j$, $j=1, 2,3,$
we find that $\partial_j \rho$ satisfies  
\begin{equation}\label{3-7-1}
     \partial_t \partial_j \rho + u \cdot \nabla \partial_j \rho  
   = -\partial_j u \cdot \nabla \rho.
\end{equation}
Then multiplying \eqref{3-7-1} by $\partial_j \rho$, integrating over
$\mathbb{R}^3$, and summing over the index $j$, one gets that  
\begin{equation} \label{3-7-2}
   \frac{d}{dt} \int |\nabla \rho|^2 \, dx  
   \leq C \int |\nabla u|\cdot |\nabla \rho|^2 \, dx 
   \leq C \|\nabla u\|_{L^\infty} \|\nabla \rho\|_{L^2}^2 ~. 
\end{equation}

To derive the appropriate bound for $\int_0^T \|\nabla u\|_{L^\infty}\, 
dt $, we make use of the elliptic estimates related with the momentum
equation. In fact,  
\begin{equation}\nonumber \ba
   \|\nabla^2 u \|_{L^4}  & \leq C \|\rho \partial_t u\|_{L^4} + C \|(\rho u\cdot \nabla) u\|_{L^4} \\
   & \leq C \|\nabla \partial_t u\|_{L^2}^{\frac34}
      \|\sqrt{\rho} \partial_t u\|_{L^2}^{\frac14} + C \|u\|_{L^{12}} \|\nabla u\|_{L^6} \\ 
   & \leq C \|\nabla \partial_t u\|_{L^2}^{\frac34}
   \|\sqrt{\rho} \partial_t u \|_{L^2}^{\frac14} + C \|\nabla u\|_{L^2}^{\frac34} \|\nabla u\|_{L^6}^{\frac54} ~,
\ea
\end{equation}
which together with Lemma~\ref{lemma-3-6} gives that 
\begin{equation}\label{3-7-4}
   \|\nabla^2 u\|_{L^2(0, T; L^4)} \leq \bar{C} ~.
\end{equation}
By virtue of the Gagliardo-Nirenberg inequality, 
\begin{equation}\nonumber
   \|\nabla u\|_{L^\infty} \leq C \|\nabla u\|_{L^6}^{\frac13} \|\nabla^2 u\|_{L^4}^{\frac23} ~,
\end{equation}
hence $\|\nabla u\|_{L^1(0, T; L^\infty)} \leq \bar{C}$. Consequently,
by  Gronwall's inequality, \eqref{3-7-2} gives a bound for $\|\nabla
\rho\|_{L^2}$.  Similar argument shows that 
\begin{equation}\label{3-7-6} \ba
   \frac{d}{dt} \int |\nabla^2 \rho |^2 \, dx  & \leq C \int \left( |\nabla u| 
     |\nabla^2 \rho|^2 + |\nabla^2 u| |\nabla \rho| |\nabla^2 \rho|
   \right) \, dx \\ 
   & \leq C \|\nabla u\|_{L^\infty} \|\nabla^2 \rho\|_{L^2}^2 
   + C \|\nabla^2 u\|_{L^3} \|\nabla \rho\|_{L^6} \|\nabla^2
   \rho\|_{L^2}  \\
   & \leq C \left( \|\nabla u\|_{L^\infty} + \|\nabla^2 u\|_{L^3}
   \right) \|\nabla^2 \rho\|_{L^2}^2 ~.
\ea
\end{equation}
Note that $\|\nabla^2 u\|_{L^3} \leq \|\nabla^2 u\|_{L^2}^{1/3}
\|\nabla^2 u\|_{L^4}^{2/3}$. Combining this fact with \eqref{3-6-7}
and \eqref{3-7-4}, we get that 
\begin{equation} \nonumber
   \nabla^2 u \in L^1(0, T; L^3) ~.
\end{equation}
Hence, the Gronwall inequality gives a bound for $\|\nabla^2 \rho\|_{L^2}$. 

Finally, using the reguality theory for Stokes equations, one can obtain that
\begin{equation}\label{3-7-7}\ba
   \|\nabla^3 u\|_{L^2} &  \leq C \left( \|\rho \partial_t u\|_{H^1} 
     + \|\rho u\cdot \nabla u\|_{H^1}  \right) \\[2mm]
   & \leq C \left(\|\nabla \rho\|_{L^3} + 1 \right)
      \left(\|\nabla \partial_t u\|_{L^2}  + \|\nabla u\|_{H^1}^2 \right), 
\ea
\end{equation}
which implies that $\|\nabla^3 u\|_{L^2(0, T; L^2)} \leq \bar{C}$. 
\end{proof}

\subsection{Proof of Theoerem \ref{main-result} }
With the a priori estimates in subsection 3.1 in hand, we are prepared
for the proof of Thorem~\ref{main-result}.
\begin{proof}
According to Theorem \ref{thm-Choe-Kim}, there exists a $T_*>0$ such
that the Cauchy problem \eqref{ns}-\eqref{initial} has a unique local
strong solution $(\rho, u, P)$ on $\mathbb{R}^3 \times (0, T_*]$,
where $T_*$ depends on $\|\rho_0\|_{L^{\frac32}\cap H^2} $, $\|\nabla
u_0\|_{H^1} $ and $\|g\|_{L^2}$. We will show that this local solution
extends to a global one.  

It follows from the integrability property of the local strong solution that $A(0)=0$. 
Hence there exists a $T_1 \in (0, T_*)$ such that \eqref{assumption} holds for $T= T_1$. Set 
\[
   T^* = \sup \{ T | \ (\rho, u, P)\  \mbox{is a strong solution on}\
     \mathbb{R}^3 \times (0, T] \ \mbox{and}\  A(T)\leq 2 \} ~. 
\]
Then $T^* \geq T_1 >0$. 

 The claim is that $T^* = \infty$, for otherwise $T^* < \infty$, which
 we will assume for an argument by contradiction. First, it
 follows from Lemmas \ref{lemma-3-6} and \ref{lemma-3-7} that 
\[
   \|\nabla u(t)\|_{H^1} \leq \bar{C}(T^*) ~, \quad 
   \|\rho(t)\|_{L^{\frac32} \cap H^2} \leq \bar{C}(T^*) ~, \quad \mbox{for
  every}\ 0<t<T^* ~,
\]
where $\bar{C}(T^*)$ depends on $T^*$ and the initial data. Secondly, 
\[
   \mu \Delta u - \nabla P = \rho \partial_t u -  (\rho u\cdot \nabla ) u, 
\]
which implies that
\be\nonumber\ba
 \left  \|\rho^{-\frac12} \left( \mu \Delta u - \nabla P \right)  \right \|_{L^2} 
   \leq  & \|\sqrt{\rho} \partial_t u  \|_{L^2} + \|( \sqrt{\rho} u
      \cdot \nabla) u\|_{L^2}  \\
   \leq &  \|\sqrt{\rho} \partial_t u  \|_{L^2} + C(\bar{\rho})
   \|\nabla u\|_{H^1}^{\frac32} \cdot \|\sqrt{\rho}
   u\|_{L^2}^{\frac12} \\
    \leq & \bar{C}(T^*) ~.
\ea
\ee
These uniform estimates allow us to construct a new unique local
strong solution from $(\rho(T_2), u(T_2))$ for any $0<T_2 < T^*$, for which
the life span is uniformly bounded from below by some $\Delta T$. By
uniqueness, on the interval $[T_2, T^*)$, this new strong solution
coincides with the original one. Choose some $T_2$ close enough to $T^*$,
such that $T_2+ \Delta T > T^*$, giving a local strong solution  
$(\rho, u, P)$ on $(0, T_2 + \Delta T]$. 

Next, we show that \eqref{assumption} holds on $(0, T_2+ \Delta T]$. 
For the strong solution $(\rho, u, P)$ on $[0, T_2+\Delta T]$, let 
\be \label{time} 
   T^{**} = \sup \left\{ T | \ \ \ A(T)\leq 2 \right\} ~.
\ee 
It follows from Theorem~\ref{keypro} that in fact
\be \nonumber
   A(T^{**})\leq 1~.
\ee
 Hence $T^{**} = T_2 + \Delta T$ and
\eqref{assumption} holds on  $(0, T_2+ \Delta T]$. In fact this is a
contradiction to the assumption that  $T^*< \infty$ is the maximal
time, since $T_2 + \Delta T > T^*$. Furthermore the decay property of
the solution is implied in the proof of Lemma \ref{lemma-3-3}. Hence, the proof is complete.  \end{proof}


\section{Proof of Theorem \ref{main-result} for the case of a bounded domain}
In this section, we assume that the domain $\Omega$ is bounded and
smooth domain in $\mathbb{R}^3$. As before, 
for strong solutions $(\rho, u , P)$
to \eqref{ns}-\eqref{initial} on $\Omega \times [0, T]$, define  
\be \label{A-4}
   A(t)= \|\nabla u\|_{L^4(0, t; \ L^2)}~,\ \ \ \ \ 0\leq t \leq T~.  
\ee

Proceeding as in the above section, we establish {\it a priori}
estimates for $(\rho, u)$ which guarantee the extension of the local solution. 
\begin{theorem}\label{prop-4-1}
Under the assumptions of Theorem \ref{thm-Choe-Kim}, there exists some
positive constant $\varepsilon$ depending only on $\mu$, $\Omega$ and 
$\bar{\rho}$, such that if $ \|u_0\|_{\dot{H}^{\frac12} }   \leq \varepsilon$ and $(\rho, u, P)$ is a
strong solution to \eqref{ns}-\eqref{initial}, satisfying  
\be\label{assumption4} 
   A(T)  \leq 2,
\ee
then it in fact holds that
\[
  A(T)  \leq 1.
\]
\end{theorem}

The proof for Theorem \ref{prop-4-1} is similar to that for
Theorem \ref{keypro}, with some slight changes due to the
presence of several additional lower order terms that appear because
of the existence of a boundary. These terms are controlled using the  
exponential decay of the kinetic energy. 

\begin{lemma} Let $(\rho, u , P)$ be a strong solution to
  \eqref{ns}-\eqref{initial}. Then for every $0\leq t \leq T$, 
  \be \nonumber
    \|\rho(t)\|_{L^\infty} = \|\rho_0\|_{L^\infty} = \bar{\rho}~.
  \ee
\end{lemma}




\begin{lemma}\label{lemma-4-3} [Modified Energy Estimate] Let 
  $(\rho, u, P)$ be a strong solution to \eqref{ns}-\eqref{initial}. 
  Then it satisfies
\be\la{4-3-1}
    \|\sqrt{\rho} u(t)\|_{L^2} \leq  \|\sqrt{\rho_0} u_0\|_{L^2} \exp
    \{ -C(\mu, \Omega, \bar{\rho}) t\} ~.
\ee
\end{lemma}

\begin{proof} According to the basic energy estimate, 
\be \la{4-3-2} 
     \frac12 \frac{d}{dt} \int \rho |u|^2 \, dx  + \mu \int |\nabla u|^2 \, dx = 0 ~. 
\ee 
By the Poincar\'e inequality, 
\[
   \int \rho |u|^2 \,  dx \leq C(\Omega, \bar{\rho}) \|\nabla u\|_{L^2}^2
   ~,
\]
which, when incorporated into \eqref{4-3-2}, yields
\be\nonumber
   \frac{d}{dt} \int \rho |u|^2 \, dx + \frac{2\mu}{C(\Omega,
     \bar{\rho})} \int \rho |u|^2 \, dx \leq 0 ~.
\ee
Consequently
\be \nonumber
   \int \rho(t) |u(t)|^2 \, dx \leq \int \rho_0 |u_0|^2 \, dx \cdot
   \exp \{-C(\mu, \Omega, \bar{\rho}) t\} ~.
\ee
\end{proof}


\begin{lemma}\label{lemma-4-3-add} Let $(\rho, u, P)$ be a strong
  solution to \eqref{ns}-\eqref{initial}. Then 
\be \la{4-3-5}
   \sup_{t\in [0, T]} t  \int \rho |u|^2 \, dx 
   + \int_0^T t \int |\nabla u|^2 \, dx dt 
   \leq C\int \rho_0 |u_0|^2 \, dx ~.
\ee
\end{lemma}

\begin{proof}
Multiplying the momentum equation by $t u$ and integrating over $\Omega$, 
\be \nonumber
   \frac{t}{2} \frac{d}{dt} \int \rho |u|^2 \, dx + \mu t
   \int |\nabla u|^2 \, dx = 0 ~,
\ee
which implies that  
\be \nonumber
   \sup_{t\in [0, T]} \frac{t}{2}  \int \rho |u|^2 \, dx  
   + \mu \int_0^T t    \int |\nabla u|^2 \, dx dt  
   = \frac12   \int_0^T   \int \rho |u|^2 \, dx dt ~.
\ee
Taking \eqref{4-3-1} into account, one finds that 
\be 
   \sup_{t\in [0, T]} t \int \rho |u|^2 \, dx 
   + \mu \int_0^T t   \int |\nabla u|^2 \, dx  dt
   \leq C \int \rho_0 |u_0|^2 \, dx ~,
\ee
which completes the proof.
\end{proof}


\begin{lemma}\la{lemma-4-4}Let $(\rho, u, P)$ be a strong solution of
  the equations \eqref{ns}-\eqref{initial}.  Suppose that
  \be \nonumber
A(T) \leq 2,
\ee
 then there exists some constant $C_1$ depending on $\mu, \bar{\rho}, \Omega$, such that
  \be \label{4-4}
    \|\nabla u\|_{L^4(0, T; L^2 )} \leq C_1 \|u_0\|_{\dot{H}^{\frac12}} ~,
  \ee
and it holds that 
\be \la{4-4-add}
\sup_{t\in [0, T]}  t \|\nabla u (t)\|_{L^2}^2 \leq C \int \rho_0 |u_0|^2 \, dx.
\ee
Note that $C_1$ does not depend on $T$.
\end{lemma}

\begin{proof}The proof is similar as in the case of the whole space.
We consider the following linear system for $(w, \tilde{P})$,  
\begin{equation}\label{Dirichlet}
  \left\{ \ba
   & \rho \partial_t w - \mu \Delta w + (\rho u\cdot \nabla )w +
   \nabla \tilde{P} = 0~,\quad \mbox{in}\ \Omega\times (0, T] ~,   \\
   & {\rm div}~w = 0 ~, \qquad \mbox{in}\ \Omega\times [0, T] ~,   \\
   & w = 0 ~, \qquad    \mbox{on}\ \partial \Omega \times [0, T]~, \\
   & w(x, 0) = w_0(x)~, \qquad \mbox{in}\ \Omega ~.
  \ea \right.
\end{equation}
Suppose $w_0$ satisfies the conditions assumed of $u_0$ in
Theorem \ref{thm-Choe-Kim}, then the global existence and uniqueness
of strong solution to \eqref{Dirichlet} is known. 
Straightforward energy estimates tell that
\be\la{energy-4}
  \frac12 \int  \rho |w(T)|^2 \, dx +  \mu \int_0^T  \|\nabla w\|_{L^2}^2 \,
  dt  \leq  \frac12 \int \rho_0 |w_0|^2 \, dx \leq C(\bar{\rho})
  \|w_0\|_{L^2}^2 ~. 
\ee

Multiplying $\eqref{Dirichlet}_1$ by $\partial_t w$ and integrating
over $\Omega$, one gets by the Sobolev embedding theorem that 
\be  \la{4-4-1} \ba
   \int \rho |\partial_t w|^2 \, dx & + \frac{\mu}{2} \frac{d}{dt}
   \int |\nabla w|^2 \, dx 
   = - \int \rho u\cdot \nabla w \cdot \partial_t w \, dx \\
   \leq & C\|u\|_{L^6} \|\nabla w\|_{L^3} \|\sqrt{\rho} \partial_t w \|_{L^2} \\
\leq &  \frac14 \int \rho |\partial_t w|^2\, dx + C \|\nabla u\|_{L^2}^2 \|\nabla w\|_{L^2} \|\nabla w\|_{H^1}~.
\ea \ee
From the regularity theory for stationary Stokes system, it follows that
\be \nonumber \ba
  & \|\nabla w\|_{H^1}  \leq C\left(\|\rho \partial_t w\|_{L^2} 
     + \|(\rho u \cdot \nabla ) w\|_{L^2} \right) + C\|\nabla w\|_{L^2} \\
&\ \ \ \leq C \|\rho \partial_t w\|_{L^2} + C \|\nabla u\|_{L^2} \|\nabla w\|_{L^2}^{\frac12} \|\nabla w\|_{H^1}^{\frac12} + C\|\nabla w\|_{L^2}~,
\ea \ee
which implies
\be \la{4-4-2}
\|\nabla w\|_{H^1} \leq C \|\rho \partial_t w\|_{L^2} + C \|\nabla u\|_{L^2}^2 \|\nabla w\|_{L^2} + C \|\nabla w\|_{L^2}~.
\ee
Inserting the estimate \eqref{4-4-2} in \eqref{4-4-1}, then 
\be \la{4-4-3} \ba
   &\int \rho |\partial_t w|^2 \, dx + \mu \frac{d}{dt}   \int |\nabla w|^2 \, dx  \\ 
& \leq C \|\nabla u\|_{L^2}^4 \|\nabla w\|_{L^2}^2 + C \|\nabla u\|_{L^2}^2 \|\nabla w\|_{L^2}^2 \\
&  \leq C \|\nabla u\|_{L^2}^4 \|\nabla w\|_{L^2}^2  + \mu \|\nabla w \|_{L^2}^2~.
\ea\ee

Utilizing Gronwall's inequality and Poincar\'e inequality,
\be \la{4-4-4} \ba
  \int_0^T \|\sqrt{\rho} &\partial_t w\|_{L^2}^2\, dt + \mu \|\nabla w (T)\|_{L^2}^2  \leq Ce^C  \|\nabla w_0\|_{L^2}^2 +  C\mu \int_0^T \|\nabla w\|_{L^2}^2\, dt \\
& \leq Ce^C  \|\nabla w_0\|_{L^2}^2 + C \int \rho_0 |w_0|^2 \, dx \leq Ce^C \|\nabla w_0\|_{L^2}^2.
\ea\ee

By interpolation, $L^4(0, T; L^2) = [L^{2}(0, T; L^2 ), \
L^\infty(0, T;  L^2)]_{\frac12, 4} $ and $\dot{H}^{\frac12} \subseteq [L^2, \dot{H}^1]_{\frac12,
  4}  $, so that  
\be \la{4-4-6}
    \|\nabla w\|_{L^4(0, T; L^2)} \leq C_1 \|w_0\|_{\dot{H}^{\frac12}} ~,
\ee
which also holds for $u$.

Furthermore, multiplying \eqref{4-4-3} by $t$ we have
\be \nonumber
     t\int \rho |\partial_t u|^2 \, dx + \mu \frac{d}{dt}
     \int t |\nabla u|^2 \, dx 
   \leq \mu \int |\nabla u|^2 \, dx + Ct \|\nabla u\|_{L^2}^6  + \mu t \|\nabla u \|_{L^2}^2~.
\ee 
By Gronwall's inequality and Lemma \ref{lemma-4-3-add}, for every $t \in [0, T]$,
\be \nonumber
\int_0^t s \int \rho |\partial_t u(s)|^2 \, dx  ds+ t\|\nabla u\|_{L^2}^2 \leq C \int_0^t  (1+s)\|\nabla u\|_{L^2}^2 \, ds 
\leq C \int \rho_0 |u_0|^2 \, dx, 
\ee
which completes the proof.

\end{proof}

The proof of Theorem~\ref{prop-4-1} is thus complete, if we 
set $\varepsilon \leq 1/C_1$. Finally, the proof of Theorem~\ref{main-result} 
for bounded domain case follows the same logic as the analog result in
section 3, we omit the details for reasons of brevity.

\end{document}